\documentclass{amsart}
\usepackage{amsmath}
\usepackage{amsthm}
\usepackage{amsfonts,mathrsfs}
\usepackage{amssymb}
\usepackage{graphicx}
\usepackage{enumitem}
\usepackage{color}
\usepackage{verbatim}
\theoremstyle{plain}
\newtheorem{theorem}{Theorem}[section]
\newtheorem{lemma}[theorem]{Lemma}
\newtheorem{corollary}[theorem]{Corollary}
\newtheorem{proposition}[theorem]{Proposition}
\newtheorem{conjecture}[theorem]{Conjecture}
\theoremstyle{definition}
\newtheorem{definition}[theorem]{Definition}

\newtheoremstyle{TheoremNum}
	{\topsep}{\topsep}              
  {\itshape}                      
  {}                              
  {\bfseries}                     
  {.}                             
  { }                             
  {\thmname{#1}\thmnote{ \bfseries #3}}
\newtheorem{remark}{Remark}
\newcommand{\F}{\mathbb F}

\newcommand{\U}{\mathcal U}
\newcommand{\Oval}{\mathcal O}

\newcommand{\Aut}{\mathrm{Aut}}

\newcommand{\Gal}{\mathrm{Gal}}
\newcommand{\Tr}{ \ensuremath{ \mathrm{Tr}}}

\newcommand{\RN}[1]{%
  \textup{\uppercase\expandafter{\romannumeral#1}}%
}

 \def\zhou#1 {\fbox {\footnote {\ }}\ \footnotetext { From Yue: {\color{red}#1}}}

 \def\trom#1 {\fbox {\footnote {\ }}\ \footnotetext { From Rocco: {\color{blue}#1}}}

\begin{document}
	\title{Unitals in shift planes of odd order}
	\author[R. Trombetti]{Rocco Trombetti}
	\address{Dipartimento di Mathematica e Applicazioni ``R. Caccioppoli", Universit\`{a} degli Studi di Napoli ``Federico \RN{2}", I-80126 Napoli, Italy}
	\email{rtrombet@unina.it}
	\author[Y. Zhou]{Yue Zhou}
	\address{Dipartimento di Mathematica e Applicazioni ``R. Caccioppoli", Universit\`{a} degli Studi di Napoli ``Federico \RN{2}", I-80126 Napoli, Italy}
	\email{yue.zhou.ovgu@gmail.com}
	\date{\today}
	\maketitle
	
	\begin{abstract}
		A finite shift plane can be equivalently defined via abelian relative difference sets as well as planar functions. In this paper, we present a generic way to construct unitals in finite shift planes of odd orders $q^2$. We investigate various geometric and combinatorial properties of them, such as the self-duality, the existences of O'Nan configurations, the Wilbrink's conditions, the designs formed by circles and so on. We also show that our unitals are inequivalent to the unitals derived from unitary polarities in the same shift planes. As designs, our unitals are also not isomorphic to the classical unitals (the Hermitian curves).
	\end{abstract}
\section{Introduction}
Let $G$ be finite group and $N$ a subgroup of $G$. A subset $D$ of $G$ is a \emph{relative difference set} with parameter $(|G|/|N|, |N|, |D|, \lambda)$ if the list of nonzero differences of $D$ comprises every element in $G\setminus N$ exactly $\lambda $ times. The subgroup $N$ is called the \emph{forbidden subgroup}. In this paper, we are interested in relative difference sets with parameters $(q,q,q,1)$ and
we write $(q,q,q,1)$-RDS for short. When $G$ is abelian, $D$ is called an \emph{abelian} $(q,q,q,1)$-RDS.

In \cite{ganley_relative_1975}, Ganley and Spence showed that, for every given $(q,q,q,1)$-RDS $D$ in $G$, we can construct an affine plane of order $q$ and group $G$ acts regularly on its affine points. Therefore we may use the elements in $G$ to denote all the affine points. All the affine lines are $D+g$ and $N+g_i$ where $g\in G$ and $\{g_i : i=1,\dots,q\}$ forms a transversal of $N$ in $G$. Clearly this affine plane can be uniquely extended to a projective plane. The extra line is $L_\infty$ and all $N+g_i$'s meet at the point $(\infty)$. It is not difficult to see that $G$ fix the flag $((\infty),L_\infty)$.

When $G$ is abelian, it is proved that $q$ has to be a power of prime. For the proofs, see \cite{ganley_paper_1976} for the $q$ even case and \cite{blokhuis_proof_2002} for the $q$ odd case. Furthermore, this abelian group $G$ and such a plane are called a \emph{shift group} and a \emph{shift plane} respectively \cite{knarr_polarities_2009}. Most of the known shift planes can be coordinatized by commutative semifields.

When $q$ is odd, all known abelian $(q,q,q,1)$-RDSs are subsets of the group $(\F_q^2,+)$. Such a $(q,q,q,1)$ RDS is equivalent to a function $f:\F_q\rightarrow\F_q$, such that $x\mapsto f(x+a)-f(x)$ is always a bijection for each nonzero $a$. This type of functions are called \emph{planar functions} on $\F_q$, which were first investigated by Dembowski and Ostrom in \cite{dembowski_planes_1968}.

As the counterpart, when $q=2^n$, abelian $(q,q,q,1)$-RDSs only exist in $C_{4}^n$ where $C_4$ is the cyclic group of order $4$. These RDSs can also be equivalently illustrated by functions over $\F_{2^n}$, which can be found in \cite{schmidt_planar_2014,zhou_2^n2^n2^n1-relative_2012}.

Let $m$ be an integer larger than or equal to $3$. A \emph{unital} of order $m$ is a $2$-$(m^3+1, m+1, 1)$ design, i.e.\ a set of $m^3+1$ points arranged into subsets of size $m+1$ such that each pair of distinct points are contained in exactly one of these subsets. 

Most of the known unitals can be embedded in a projective plane $\Pi$ of order $q^2$. In such cases, the \emph{embedded unital} is a set $\U$ of $q^3+1$ points such that each line of $\Pi$ intersects $\U$ in $1$ or $q+1$ points. When $\Pi$ is the desarguesian projective plane $\mathrm{PG}(2,q^2)$, the set of absolute points of a unitary polarity, or equivalently speaking, the rational points on a nondegenerate Hermitian curve form a \emph{classical} unital. There are also non-classical unitals in $\mathrm{PG}(2,q^2)$, for instance the Buekenhout-Metz unitals \cite{buekenhout_characterizations_1976}. There also exist unitals which can not be embedded in a projective plane, such as the Ree unitals \cite{luneburg_remarks_1966}.

Similarly as in desarguesian planes, unitals can be derived from unitary polarities in shift planes; see \cite{abatangelo_ovals_1999,abatangelo_polarity_2002,ganley_class_1972,ganley_polarities_1972,hui_non-classical_2013,knarr_polarities_2010} for their constructions and related research problems. As a special type of shift planes, commutative semifield planes also contain the unitals which are analogous to the Buekenhout-Metz ones in desarguesian planes; see \cite{abatangelo_transitive_2001,zhou_parabolic_2015}.

In this paper, we consider the unitals in the shift planes of order $q^2$. We restrict ourselves to the $q$ odd case, because most of the calculations and constructions in the $q$ even case are quite different.

The organization of this paper is as follows: In Section \ref{sec:construction}, after a brief introduction to shift planes $\Pi(f)$ derived from planar functions $f$, we present a generic construction of unitals in shift planes of order $q^2$. In Section \ref{sec:properties}, we investigate various geometric and combinatorial properties of them, such as the self-duality, the existences of O'Nan configurations, the Wilbrink's conditions, the designs formed by circles and so on. In Section \ref{sec:inequivalence}, we consider the unitals derived from the unitary polarities of the shift planes and show that they are not equivalent to the unitals constructed in Section \ref{sec:construction}. As designs, our unitals are also not isomorphic to the classical unitals (the Hermitian curves).

\section{Construction of Unitals}\label{sec:construction}
Let $\F$ be a finite field of odd order and $f$ a planar function on $\F$. The projective plane $\Pi(f)$ derived from $f$ is defined as follows:
\begin{itemize}
	\item \textbf{Points:} $(x,y)\in \F\times\F$ and $(a)$ for $a\in \F\cup\{\infty\}$;
	\item \textbf{Lines:} $L_{a,b}:=\{(x,f(x+a)-b): x\in\F\}\cup \{(a)\}$ for all $(a,b)\in\F\times \F$, $N_a := \{(a,y): y\in\F \}\cup \{(\infty)\}$ and $L_\infty:= \{(a): a\in \F\cup \{\infty\} \}$.
\end{itemize}

The points except for those on $L_\infty$ are called the \emph{affine points} of $\Pi(f)$. By removing the line $L_\infty$ and the points on it, we get an affine plane.

It is routine to verify that the set of maps 
\[T:=\{\tau_{u,v}: \tau_{u,v}(x,y)=(x+u,y+v): u,v \in \F\}\] 
induces a collineation group on $\Pi(f)$ and this group is abelian and acts regularly on the affine points and all lines $\{L_{a,b}: a,b\in \F\}$. It is also transitive on the line set $\{N_a : a \in \F\}$ and on the points of $L_\infty\setminus\{(\infty)\}$. Hence there are three orbits of all the points (lines) in $\Pi(f)$ under this group. We call it the \emph{shift group} of $\Pi(f)$.

When $f$ can be written as a Dembowski-Ostrom polynomial, i.e.\ $f(x)=\sum a_{ij}x^{p^i+p^j}$ where $p=\mathrm{char}(\F)$, the plane $\Pi(f)$ is also a commutative semifield plane. Using the corresponding semifield multiplication, we can label the points and lines of $\Pi(f)$ in a different way. The intersection of the translation group and the shift group of $\Pi(f)$ is $\{(x,y)\mapsto (x,y+b): b\in \F\}$. See \cite[Section 4]{ghinelli_finite_2003} for details. We refer to \cite{lavrauw_semifields_2011} and \cite{pott_semifields_2014} for recent surveys on semifields and relative difference sets respectively.

Up to equivalence, all known planar functions $f$ on finite fields $\F_q$ of odd characteristics can be written as a Dembowski-Ostrom polynomial except for the Coulter-Matthews ones which are power maps defined by $x\mapsto x^d$ on $\F_{3^m}$ for certain $d$; see \cite{coulter_planar_1997}. Both the Dembowski-Ostrom planar functions and the Coulter-Matthews ones satisfy that
\begin{itemize}
	\item $f(0)=0$ and 
	\item for arbitrary $a,b\in \F_q$, $f(a)=f(b)$ if and only if $a=\pm b$.
\end{itemize}
For a proof of the Dembowski-Ostrom polynomials case, we refer to \cite{kyureghyan_theorems_2008}; for the Coulter-Matthews functions $f(x)=x^d$ on $\F_{3^m}$, it can be verified directly from the fact $\gcd(d,3^m-1)=2$. Actually for a function $f$ defined by a Dembowski-Ostrom polynomial, the above conditions are necessary and sufficient for $f$ to be planar; see \cite{weng_further_2010}. If a planar function satisfies the aforementioned two conditions, then we call it a \emph{normal planar function}.

\begin{lemma}\label{le:basic}
	Let $f$ be a planar function on $\F_{q^2}$, where $q^2$ is odd. Let $\{g_a: a\in \F_{q^2}\}$ be a set of injections from $\F_q$ to $\F_{q^2}$. The set of points
	\[\U_g:=\{(x,g_x(t)): x\in\F_{q^2}, t\in \F_q \} \cup \{(\infty)\} \]
	is a unital in $\Pi(f)$ if for each $a,b\in\F_{q^2}$, there are $1$ or $q+1$ pairs $(x,t)$ such that
	\begin{equation}\label{eq:general}
		f(x+a)- b-g_x(t)=0.
	\end{equation}
\end{lemma}
\begin{proof}
	It is not difficult to verify that $\U_g$ satisfies the definition of a unital in the following steps. First, every line through $(\infty)$ meet $\U_g$ at $q+1$ points. Second, $L_\infty$ meet it at $(\infty)$. Finally, the number of common points of $L_{a,b}$ and $\U_g$ are exactly the number of pairs of $(x,t)$ such that $f(x+a) - b-g_x(t)=0$.
\end{proof}

Let $\xi$ be an element in $\F_{q^2}\setminus \F_q$. Then every element $x$ of $\F_{q^2}$ can be written as $x=x_0+x_1\xi$ where $x_0,x_1\in\F_q$. Similarly, every function $f:\F_{q^2}\rightarrow \F_{q^2}$ can be written as  $f(x)=f_0(x)+f_1(x)\xi$ where $f_0,f_1$ are maps from $\F_q$ to itself. Throughout this paper, we frequently switch between the element $x\in \F_{q^2}$ and its two dimensional representation $(x_0,x_1)\in \F_{q}^2$. If a special assumption on $\xi$ is needed, we will point it out explicitly.

\begin{proposition}\label{prop:general}
	Let $f$ be a planar function on $\F_{q^2}$ and  $\theta\in\F_{q^2}^*$, where $q$ is odd. Assume that 
	\[\#\{x\in\F_{q^2}:\theta_1 f_0(x) - \theta_0 f_1(x)=c\} =
	\left\{
	  \begin{array}{ll}
	    q+1, & \hbox{$c\neq 0$;} \\
	    1, & \hbox{$c=0$.}
	  \end{array}
	\right.
	\]
	Then the set of points
	\begin{equation}\label{eq:u_theta_general}
		\U_\theta:=\{(x,t\theta): x\in\F_{q^2}, t\in \F_q \} \cup \{(\infty)\} 
	\end{equation}
	is a unital in $\Pi(f)$. Furthermore, $L_{a,b}$ is a tangent line to $\U_\theta$ if and only if $b_0\theta_1-b_1\theta_0=0$.
\end{proposition}
\begin{proof}
	First, \eqref{eq:general} is equivalent to the following two equations
	\begin{eqnarray}
		\label{eq:square} f_0(x) - b_0 &=& t\theta_0,\\
		\nonumber f_1(x) - b_1 &=& t\theta_1.
	\end{eqnarray}
	As $\theta\neq0$, the system of equations above is equivalent to
	\begin{eqnarray}
	\label{eq:general.first}	\theta_1f_0(x)-\theta_0f_1(x)-(\theta_1b_0-\theta_0b_1) &=& 0,\\
	\nonumber				f_1(x) - b_1 &=& t\theta_1.
	\end{eqnarray} 
	(If $\theta_1=0$, then we replace the second equation by $f_0(x)-b_0=t\theta_0$.) Noting that $t$ ranges through all the elements in $\F_q$, we see that the cardinality of $(x,t)$ satisfying \eqref{eq:general} is equivalent to the cardinality of $x$ such that \eqref{eq:general.first} holds. According to the assumption, $\U_\theta$ is a unital in $\Pi(f)$. The last statement of the proposition follows from \eqref{eq:general.first} directly.
\end{proof}

Next we consider several special cases of Proposition \ref{prop:general}. Let $\eta$ be the quadratic character on $\F_q$, i.e.\
\begin{equation*}
	\eta(x) := \left\{
	  \begin{array}{rl}
	    1, & \hbox{if $x$ is a square;} \\
	    -1, & \hbox{if $x$ is not a square;} \\
	    0, & \hbox{if $x=0$.}
	  \end{array}
	\right.
\end{equation*} 
The integer-valued function $\nu$ on $\F_q$ is defined by $\nu(b)=-1$ for $b\in \F_q^*$ and $\nu(0)=q-1$. 
To prove the existence of unitals $\U_\theta$ in many shift planes, we need the following well-known result.
\begin{lemma}\label{le:quadratic}
	For odd prime power $q$, let $b$, $a_0$, $a_1$ and $a_2\in\F_q$. Let 
	\[Q(x_0,x_1):=a_0x_0^2+a_1x_0x_1+ a_2x_1^2.\]
	Its \emph{discriminant} is defined by $\Delta:=a_0a_2-a_1^2/4$.
	Assume that $\Delta\neq 0$. Then the number of solution of $a_0x_0^2+a_1x_0x_1+a_2x_1^2=b$ is
	\[N(Q(x_0,x_1)=b)=q+\nu(b)\eta(-\Delta).\]
	In particular, when $-\Delta$ is a nonsquare, 
	\[N(Q(x_0,x_1)=b)=\left\{
	  \begin{array}{ll}
	    q+1, & \hbox{$b\neq 0$;} \\
	    1, & \hbox{$b=0$.}
	  \end{array}
	\right.
	\]
\end{lemma}

\begin{theorem}\label{th:many.planar.functions.1}
	Let $f$ be a planar function on $\F_{q^2}$, where $q^2=p^{2n}$, $p$ is an odd prime and $n$ is a positive integer. Assume that $\theta$ is an element in $\F_{q^2}^*$ such that $\theta^{q+1}$ is a nonsquare in $\F_q$ and $f$ satisfies that
	\begin{equation}\label{eq:coro.condition.1}
	\#\{x\in\F_{q^2}:f(x)=c\} =\#\{y \in \F_{q^2}: y^2=c\},
	\end{equation}
	for each $c\in\F_{q^2}$. Then the set of points
		\[\U_\theta:=\{(x,t\theta): x\in\F_{q^2}, t\in \F_q \} \cup \{(\infty)\} \]
	is a unital in $\Pi(f)$.
	Furthermore, $\U_\theta$ is a unital for the following planar functions:
	\begin{enumerate}[label=(\alph*)]
		\item $f(x)=x^2$, i.e.\ $\Pi(f)$ is a Desarguesian plane.
		\item $f(x)=x^{p^k+1}$ where $k$ is an integer satisfying that $1\le k\le n$ and $2n/\gcd(2n,k)$ is odd, i.e.\ $\Pi(f)$ is an Albert's commutative twisted field plane \cite{albert_generalized_1961}.
		\item $f(x)=x^{\frac{3^k+1}{2}}$ where $\gcd(k,2n)=1$ (now $p=3$), i.e.\ $\Pi(f)$ is a Coulter-Matthews plane which is not a translation plane \cite{coulter_planar_1997}.
	\end{enumerate}
\end{theorem}
\begin{proof}
	Under the assumption, it is clear that
	\[\#\{x:f(x+a)-b=t\theta \} = \#\{y:y^2-b=t\theta \}.\]
	Now assume that $\xi\in\F_{q^2}\setminus\F_q$ satisfying $\xi^2=\alpha\in\F_q$. Similarly as in Proposition \ref{prop:general}, we denote $y=y_0+y_1\xi$ and $\theta=\theta_0+\theta_1\xi$ where $y_0,y_1,\theta_0$, $\theta_1\in\F_q$ and $(\theta_0,\theta_1)\neq (0,0)$. Then
	\[y^2 = y_0^2 + \alpha y_1^2 + 2y_0y_1 \xi.\]
	By \eqref{eq:general.first} in Proposition \ref{prop:general} and Lemma \ref{le:quadratic}, we only have to show that $\theta_1y_0^2-2\theta_0y_0y_1+\theta_1\alpha y_1^2$ is an irreducible quadratic form. Its discriminant equals $\theta_1^2\alpha-\theta_0^2=-\theta^{q+1}$. From the assumption and Lemma \ref{le:quadratic}, it follows that $\U_\theta$ is a unital in $\Pi(f)$.
	
	As the three families of planar functions are all power maps, by considering the greatest common divisors of the exponent and $q^2-1$, it is not difficult to show that they all satisfy \eqref{eq:coro.condition.1}. Therefore, $\U_\theta$ is a unital in anyone of these shift planes.
\end{proof}


\begin{theorem}\label{th:many.planar.functions}
	Let $f$ be a planar function on $\F_{q^2}$, where $q^2=p^{2n}$, $p$ is an odd prime and $n$ is a positive integer. Let $\xi$ be an element in $\F_{q^2}\setminus \F_q$. Let $\U_\xi$ be the set of points defined by \eqref{eq:u_theta_general}.
	
	When $p^n \equiv 1 \pmod 4$ and $\alpha$ is a nonsquare in $\F_q$, $\U_\xi$ is a unital in the following commutative (pre)semifield planes, besides those appeared in Theorem \ref{th:many.planar.functions.1}:
	\begin{enumerate}[label=(\alph*)]
		\item Dickson's semifield planes \cite{dickson_commutative_1906} and the corresponding planar functions are
		\[f(x)=(x_0^2+\alpha x_1^{2p^i})+2x_0x_1\xi ,\]
		where $i$ is an integer satisfying $0<i<n$.
		
		\item The semifield planes constructed by Pott and the last author in \cite{zhou_new_2013} and the corresponding planar functions are
		\[f(x)=(x_0^{p^k+1}+\alpha x_1^{p^{k+i}+p^i})+2x_0x_1\xi ,\]
		where $i,k$ are integers such that $0<i,k<n$ and $n/\gcd(k,n)$ is odd.
	\end{enumerate}
	
	When $p=3$, $\U_\xi$ is a unital in 
	\begin{enumerate}[label=(\alph*)]\setcounter{enumi}{2}
	\item 	Ganley's semifields planes where $n$ is defined to be odd \cite{ganley_central_1981}. The corresponding planar functions are
		\[f(x)=(x_0^2+x_1^{10})+(2x_0x_1+x_1^6)\xi.\]
	\item Penttila-Williams semifield planes where $n=5$ \cite{penttila_ovoids_2004}. The corresponding planar functions are
			\[f(x)=(x_0^2+x_1^{18})+(2x_0x_1+x_1^{54})\xi.\]
	\end{enumerate}
	
	When $p^n\equiv 3 \pmod{4}$, $\U_\xi$ is a unital in
	\begin{enumerate}[label=(\alph*)]\setcounter{enumi}{4}
		\item Budaghyan-Helleseth semifield planes \cite{budaghyan_new_2008}. The corresponding planar functions are
		\[f(x)= bx^{p^k+1}+(bx^{p^k+1})^{p^n}+\xi x^{p^n+1},\]
		where integer $k$ is such that $0<k<n$ and $2\nmid \frac{n}{\gcd(k,n)}$, and $b$ is a nonsquare in $\F_{q^2}^*$.
	\end{enumerate}
\end{theorem}
\begin{proof}
	The proof for the first two cases is straightforward: As the first component of the corresponding planar functions are $x_0^2+\alpha x_1^{2p^i}$ and $x_0^{p^k+1}+\alpha x_1^{(p^k+1)p^i}$, clearly we can choose $Q(y_0,y_1)=y_0^2+\alpha y_1^{2}$ for these two cases and $Q$ is irreducible if and only if $-\alpha$ is a nonsquare in $\F_{p^n}$. It is equivalent to that $-1$ is a square, which holds exactly when $p^n \equiv 1\pmod 4$.
	
	The cases (c) and (d) are also not difficult to verify: We can replace $x_1^5$ in $x_0^2+x_1^{10}$ by $y_1$ and replace $x_1^9$ in $x_0^2+x_1^{18}$ by $y_1$. Then the results follow from the fact that $-1$ is a nonsquare in $\F_{3^n}$ where $n$ is odd.
	
	In (e), noting that $f(x)=f_0(x)+f_1(x)\xi$ where $f_0(x)=bx^{p^k+1}+(bx^{p^k+1})^{q}\in\F_{q}$ and $f_1(x) = x^{q+1}\in\F_{q}$ for $x\in \F_{q^{2}}$, by Proposition \ref{prop:general} we only have to concentrate on $f_0$. As $\gcd(p^k+1, p^n-1)=2$, we see that
	\[\#\{x : f_0(x)=a\}=\#\{y: g(y):=by^2+(by^2)^{q}=a\}, \]
	for each $a\in\F_{q}$.
	
	Now we choose $\varepsilon\in\F_{q^{2}}\setminus\F_{q}$ satisfying that $\varepsilon^{q}+\varepsilon=0$ and write $y=y_0+y_1\varepsilon$ where $y_0,y_1\in\F_{q}$. It is clear that $\Tr_{q^{2}/q}(y)=y_0$. Then 
	\begin{align*}
		g(y)& = \Tr_{q^{2}/q}(by^2)\\
			& = \Tr_{q^{2}/q}((b_0+b_1\varepsilon)(y_0^2+\varepsilon^2 y_1^2+ 2y_0y_1\varepsilon))\\ 
			& = 2b_0(y_0^2+\varepsilon^2 y_1^2)+4\varepsilon^2 b_1y_0y_1.
	\end{align*}
	Its discriminant equals $\Delta=2b_0^2 \cdot 2\varepsilon^2-4\varepsilon^4b_1^2=4\varepsilon^2b\cdot b^{q}$. Clearly $\varepsilon^2$ is a nonsquare in $\F_{q}$, otherwise $\varepsilon\in\F_{q}$ leading to a contradiction. As $b$ is a nonsquare in $\F_{q^{2}}$, $b\cdot b^{q}$ has to be a nonsquare in $\F_{q}$. Therefore $\eta(-\Delta)$ is a nonsquare in $\F_{q}$ if and only if $-1$ is a nonsquare in it. This happens exactly when $q=p^n \equiv 3 \pmod 4$. From Proposition \ref{prop:general} and Lemma \ref{le:quadratic}, it follows that $\U_\xi$ is a unital.
\end{proof}

\section{Properties of $\U_\theta$}\label{sec:properties}
In this section, we proceed to investigate several common properties of $\U_\theta$ obtained in Proposition \ref{prop:general}. Several important subgroups of their automorphism groups and self-duality are considered in Subsection \ref{subsec:basic}. Then we look at the projections of the blocks in the unital, which are called circles, and the Wilbrink condition \RN{2}. In Subsection \ref{subsec:onan}, we investigate the existence of O'Nan configurations. To conclude this section, we present a conjecture on the automorphism groups of our unitals and several comments on this conjecture.

\subsection{Basic properties}\label{subsec:basic}

An \emph{oval} $\Oval$ in a projective plane $\Pi$ of odd order $q$ is a set of $q+1$ points such that every line in $\Pi$ meets $\Oval$ in $0$, $1$ or $2$ points. According to the famous result by Segre in \cite{segre_ovals_1955}, all ovals in desarguesian planes of odd orders are nondegenerate conics.
\begin{proposition}\label{prop:unital=ovals}
	Let $f$ be a normal planar function on $\F_{q^2}$ and $\U_\theta$ be a unital in $\Pi(f)$ constructed in Proposition \ref{prop:general}. Then for each $c\in \F_{q^2}$, the set 
	\[\Oval_c:=\{(x, c): x \in \F_{q^2}\}\cup \{(\infty)\}\]
	is an oval in $\Pi(f)$ and $\U_\theta$ is a union of ovals.
\end{proposition}
\begin{proof}
	As $f$ is normal, for each given $b\in \F_{q^2}$, there are at most two solution such that $f(x)=b$. By checking the cardinalities of the points in $\Oval_c\cap L_{a,b}$, $\Oval_c\cap N_a$ and $\Oval_c\cap L_\infty$, we see that $\Oval_c$ is an oval in $\Pi(f)$. Therefore $\U_\theta= \bigcup_{t\in \F_q} \Oval_{t\theta}$.
%
\end{proof}

\begin{remark}
	\begin{enumerate}[label=(\alph*)]
	\item A family of unitals in $PG(2,q^2)$, each of which is a union of ovals, were independently discovered by Hirschfeld and Sz\"{o}nyi \cite{hirschfeld_sets_1991} and by Baker and Ebert \cite{baker_intersection_1990}. Actually, Baker and Ebert \cite{baker_buekenhout-metz_1992} also showed that this family of unitals is a special subclass of the Buekenhout-Metz ones in $PG(2,q^2)$. Noting that $\Pi(f)$ is desarguesian \cite{dembowski_planes_1968}, we can readily verify that our unitals $\U_\theta$ in $\Pi(x^2)$ (see Theorem \ref{th:many.planar.functions.1} (a)) are exactly the family of unitals obtained in \cite{baker_intersection_1990,hirschfeld_sets_1991}. This result can also be derived from Proposition \ref{prop:unital=ovals} and \cite[Theorem 1.1]{durante_unitals_2013}.
	
	An analogous result for Albert's twisted field planes (Theorem \ref{th:many.planar.functions.1} (b)) can be found in \cite{abatangelo_transitive_2001} by Abatangelo, Korchm\'{a}ros and Larato.

	\item Bukenhout-Metz type of unitals can be also obtained for Dickson's semifield planes and the ones constructed in \cite{zhou_new_2013}. We refer to \cite{zhou_parabolic_2015}, in which the unitals obtained in Theorem \ref{th:many.planar.functions} (a) and (b) appear as a subclass of a family of Bukenhout-Metz type  of unitals.
	\end{enumerate}
\end{remark}

Given two unitals $\U_1$ and $\U_2$, we say that they are \emph{isomorphic} if there is a design isomorphism between them, i.e.\ there is a bijection between their point sets which maps the blocks of $\U_1$ to the blocks of $\U_2$. When $\U_1$ and $\U_2$ can be embedded into the same projective plane $\Pi$, we say they are \emph{equivalent} if there is a collineation of $\Pi$ mapping $\U_1$ to $\U_2$. For a unital $\U$, we use $\Aut(\U)$ to denote its automorphism group. When $\U$ can be embedded in a projective plane $\Pi$, we denote the group of all the collineation of $\Pi$ fixing $\U$ by $\Aut_\Pi(\U)$.

\begin{proposition}\label{prop:unital_powerplanar_orbit}
	Let $d$ be a positive integer and $f(x)=x^d$ a planar function on $\F_{q^2}$. Let $p=\mathrm{char}(\F_q)$. For $c\in\F_{q^2}^*$ and $\sigma\in \Gal(\F_q/\F_p)$,  $\gamma_{c,\sigma}$ on the points of $\Pi(f)$ is defined by
	\begin{eqnarray*}
		(x,y) &\mapsto& (\sigma(cx),\sigma(c^dy)),\\
		(a)&\mapsto& (\sigma(ca)),\\
		(\infty)&\mapsto& (\infty).
	\end{eqnarray*}
	Then all $\gamma_{c,\sigma}$ together form a group $\Gamma$ as a collineation group of $\Pi(f)$. In $\Pi(f)$, all the unitals $\U_{\theta}$ constructed in Theorem \ref{th:many.planar.functions.1} are equivalent.
\end{proposition}
\begin{proof}
	First it is straightforward to check that under $\gamma_{c,\sigma}$, the line $L_{a,b}$ is mapped to $L_{\sigma(ca), \sigma(c^db)}$, the line $N_a$ is mapped to $N_{\sigma(ca)}$ and $L_\infty$ is fixed. Hence $\gamma_{c,\sigma}$ is a collineation. It is also not difficult to see that under the composition of maps, the elements in $\Gamma$ form a group.

	Let $\theta$ be an element in $\F_{q^2}^*$ satisfying that $\theta^{q+1}$ is a nonsquare in $\F_{q}$ (see Theorem \ref{th:many.planar.functions.1}), i.e.\ $\theta$ is an odd power of a primitive element of $\F_{q^2}$. Applying $\gamma_{c,\sigma}$ on $\U_\theta$, we get $\U_{\sigma(c^d\theta t)}$. As $\sigma(c^d\theta t)=(\sigma(c^dt)\frac{\sigma(\theta)}{\theta}) \cdot \theta$ and the set $\left\{\sigma(c^dt)\frac{\sigma(\theta)}{\theta} : c\in \F_{q^2}^* \right\}$ covers all the nonzero squares in $\F_{q^2}$, we see that all the unitals $\U_{a^2\theta}$ for $a\in \F_{q^2}^*$ are in one orbit under $\gamma_{c,\sigma}$. Hence all the unitals $\U_{\theta'}$, where $\theta'$ satisfies the condition in Theorem \ref{th:many.planar.functions.1}, are equivalent.
\end{proof}

Let $\U$ be a unital of order $n$ embedded in a projective plane $\Pi$. We can obtain a new design $\U^*$ in the dual plane $\Pi^*$ by taking the tangent lines to $\U$ as the points of $\U^*$ and the points of $\Pi\setminus \U$ as the blocks of $\U^*$. The incidence of $\U^*$ is given by reverse containment. It is easy to see that $\U^*$ is another $2$-$(n^3+1,n+1,1)$ design, which is called the \emph{dual unital} of $\U$. If $\U$ and $\U^*$ are isomorphic as designs, then $\U$ is called \emph{self-dual}.

Let $f$ be a planar function and $\Pi(f)$ the corresponding shift plane. Let $\Pi(f)^*$ be its dual plane. 
The points in $\Pi(f)^*$ are the three types of lines $L_{a,b}$, $N_a$ and $L_\infty$ in $\Pi(f)$. 

If $y=f(x+a)-b$, then clearly $b=f(a+x)-y$. That means the line $L_{x,y}$ in $\Pi(f)$ contains $(a,b)$ if and only if the point $L_{a,b}$ in $\Pi(f)^*$ is on the line $(x,y)$. For $N_a$'s and $L_\infty$, we can also get similar results. That means if we switch the notations of points and lines in $\Pi(f)$, i.e.
\begin{align*}
	(a,b) &\leftrightarrow L_{a,b},\\
	(a)  &\leftrightarrow N_a,\\
	(\infty) &\leftrightarrow L_\infty,
\end{align*}
then we get $\Pi(f)^*$.
\begin{proposition}
	Let $\U_\theta$ be the unital of order $q$ in $\Pi(f)$ defined in Proposition \ref{prop:general} where $f$ is a planar function. Then $\U_\theta$ is self-dual.
\end{proposition}
\begin{proof}
	By Proposition \ref{prop:general}, we know that $L_{a,b}$ is a tangent line if and only if $b_0\theta_1=b_1\theta_0$, i.e.
	\[b= \left\{
	  \begin{array}{ll}
	    \frac{b_0}{\theta_0}\theta, & \hbox{$\theta_0\neq0$;} \\
	    \frac{b_1}{\theta_1}\theta, & \hbox{$\theta_1\neq1$.}
	  \end{array}
	\right.
	\]
	As the constant $\theta\neq 0$, $b=s \theta$ for certain $s\in \F_q$. Together with the tangent line $L_\infty$, we know that the dual of $\U_\theta$ in $\Pi(f)^*$ can be written as 
	\[\U_\theta^*=\{L_{x,t\theta}: x\in\F_{q^2}, t\in \F_q \} \cup \{L_\infty\}.\]
	Switching the notations of points and lines in $\U_\theta^*$ and $\Pi(f)^*$, we get $\U_\theta$ and $\Pi(f)$. Therefore $\U_\theta$ is self-dual. 
\end{proof}

\subsection{Circles and Wilbrink' condition \RN{2}}\label{subsec:circles}
Let $f$ be a planar function on $\F_{q^2}$ and $\Pi(f)$ the corresponding projective plane. Let $\U_\theta$ be a unital in $\Pi(f)$ defined by \eqref{eq:u_theta_general}. As a design, its point set is $\{(x,t\theta): x \in \F_{q^2}, t\in \F_q\}\cup \{(\infty)\}$ and all of its blocks are
\[B_{a}:= \{(a,t\theta): t\in \F_q\}\cup \{(\infty) \},\]
for each $a\in \F_{q^2}$ and
\[B_{a,b}:=\{(x,t\theta): f(x+a)-b=t\theta, t\in \F_q \}, \]
for each $a,b\in \F_{q^2}$ where $b_0\theta_1-b_1\theta_0\neq 0$. There are totally $q^4-q^3+q^2$ blocks and each of them contains $q+1$ points. For each pair of points, there is exactly one block containing them both. Hence $\U_\theta$ is a $2$-$(q^3+1,q+1,1)$-design.

We denote 
\begin{equation}\label{eq:definition_circles}
	C_{a,\beta(b)}:=\{x : f_0(x+a)\theta_1-f_1(x+a)\theta_0=\beta(b)\},
\end{equation}
where $\beta(b)=b_0\theta_1-b_1\theta_0\neq 0$. Clearly, $C_{a,\beta(b)}$ is the set of the elements appearing in the first coordinate of the elements in $B_{a,b}$; in other words, $C_{a,\beta(b)}$ can be viewed as a projection of $B_{a,b}$. When the context is clear, we omit $b$ and write $C_{a,\beta(b)}$ as $C_{a,\beta}$. Inspired by the approach of O'Nan in \cite{onan_automorphisms_1972}, we call $C_{a,\beta}$ a \emph{circle} for $a\in \F_{q^2}$ and $\beta\in \F_q^*$. By setting $\phi(x):=f_0(x)\theta_1-f_1(x)\theta_0$, we can shortly write a circle as
\[C_{a,\beta}=\{x: \phi(x+a)=\beta\}.\]
By choosing appropriate $\delta\in \F_{q^2}^*$, we can also write 
\[\phi(x)=\Tr_{q^2/q}(\delta f(x)). \]
\begin{lemma}\label{le:circles}
	Let $\mathscr{C}_\theta$ denote the set of circles derived from $\U_\theta$. Then following statements hold.
	\begin{enumerate}[label=(\alph*)]
		\item $\#C_{a,\beta}=q+1$, for arbitrary $a\in \F_{q^2}$ and $\beta\in \F_q^*$.
		\item For each $a\in \F_{q^2}$, the set $\{C_{a,\beta}: \beta\in \F_{q}^* \}$ forms a partition of $\F_{q^2}\setminus\{-a\}$.
		\item All the blocks $B_{a,b}$ which intersect $B_{u}$ are projected to $C_{a,\phi(u+a)}$.
		\item $C_{a,\beta}=C_{a',\beta'}$ if and only if $a=a'$ and $\beta=\beta'$.
		\item $\# \mathscr{C}_\theta = q^3-q^2$.
	\end{enumerate}
\end{lemma}
\begin{proof}
	(a) and (b) follows directly from  the definition of $C_{a,\beta}$. 
	
	(c). When $B_{a,b}$ intersects $B_{u}$, that means $f(u+a)-b=s\theta$ for certain $s\in \F_q$. Together with \eqref{eq:general.first}, we have
	\[\phi(u+a)=f_0(u+a)\theta_1-f_1(u+a)\theta_0=b_0\theta_1-b_1\theta_0=\beta(b).\] Hence $B_{a,b}$ is projected to $C_{a,\phi(u+a)}$.
	
%
	(d). Assume that $\beta$ and $\beta'$ are both not $0$, $C_{a,\beta}=C_{a',\beta'}$ and $a\neq a'$. It follows that there are at least $q+1$ elements $x$ in $\F_{q^2}$ such that
	\[\phi(x+a')-\phi(x+a)=\beta' - \beta,\]
	which equals
	\begin{equation}\label{eq:circles.same}
		\Tr_{q^2/q}(\delta (f(x+a')-f(x+a)))=\beta'-\beta.
	\end{equation}
	As there are totally $q$ elements $c$ in $\F_{q^2}$ such that $\Tr_{q^2/q}(c)=\beta'-\beta$ and the map defined by $x\mapsto f(x+a')-f(x+a)$ is a permutation, there are exactly $q$ elements such that \eqref{eq:circles.same} holds. It is a contradiction. Hence $a=a'$. From (b), we immediately get that $\beta=\beta'$.
	
	(e) follows from (d) using a simple counting argument.
\end{proof}
\begin{corollary}
	$(\F_{q^2}, \mathscr{C}_\theta)$ is a $(q^2,q+1,q)$-design.
\end{corollary}
\begin{proof}
	We only have to prove that for each two difference elements $u,v\in\F_{q^2}$, there are exactly $q$ circles containing them. 
	
	By Lemma \ref{le:circles} (c), the circle $C_{a,\beta}$ contains $u$ and $v$ if and only if $\beta=\phi(u+a)=\phi(v+a)$, which is equivalent to
	\[\Tr_{q^2/q}(\delta(f(u+a)-f(v+a))) = 0.\]
	By the same argument in the proof of Lemma \ref{le:circles} (d), there are exactly $q$ elements $a\in \F_{q^2}$ such that the above equation holds. Therefore there are exactly $q$ circles containing $u$ and $v$.
\end{proof}

In \cite{wilbrink_characterization_1983}, Wilbrink characterized the classical unital by three intrinsic conditions. The Wilbrink's condition \RN{2} on a unital $\U$ is as follows.
\begin{definition}\label{de:wilbrink}
	Let $v$ be a point of a unital $\U$. Let $B$ be a block such that $v\not \in B$. Let $C$ be a block which contains  $v$ and meets $B$, and $w$ a point in $C$ which is distinct from $v$ and $B\cap C$. If there exists a block $B'\neq C$ such that $w\in B'$ and $B'$ meets all blocks which contain $v$ and meet $B$, then  we call $v$ a \emph{vertex of Wilbrink's condition \RN{2}}; see Figure \ref{fig:wilbrink}. If $v$ is such a vertex for all the blocks $C$, $B$ and point $w$ in $C$ satisfying aforementioned conditions, then $v$ is called a  \emph{vertex of Wilbrink's condition \RN{2} in strong form} \cite{hui_non-classical_2013}.
\end{definition}

\begin{figure}[t]
\centering
\includegraphics[scale=0.45]{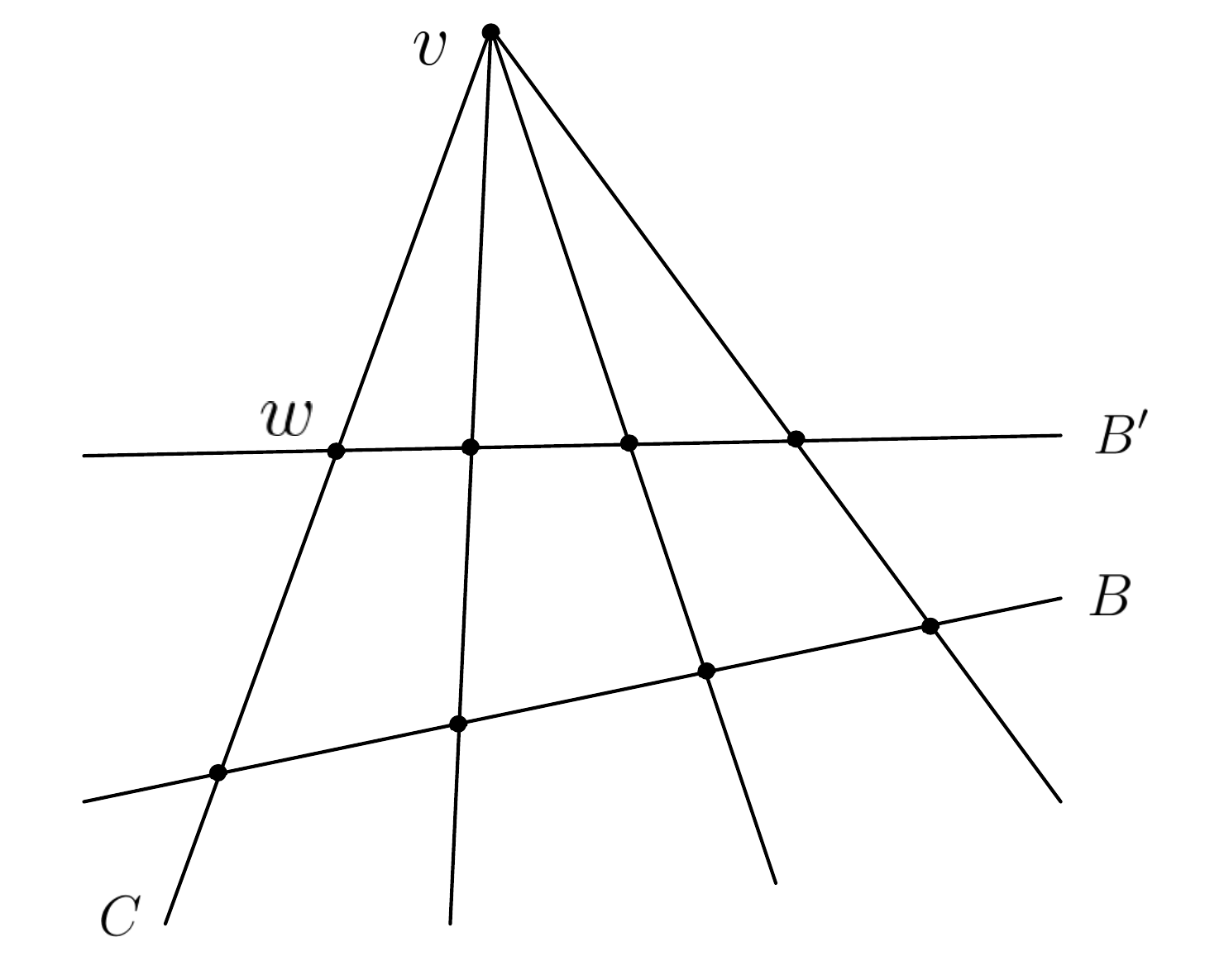}
\caption{Wilbrink's condition \RN{2}}\label{fig:wilbrink}
\end{figure}

A point $v$ in $\U$ is a vertex of Wilbrink's condition \RN{2} in strong form together with the nonexistence of an O'Nan configuration (see Subsection \ref{subsec:onan}) containing $v$ imply that $B$ and $B'$ in above definition have no common points. Hence there is a parallelism defined on the blocks of $\U$ not containing $v$, i.e.\ two blocks $B$ and $B'$ of $\U$ are parallel if they meet the same blocks containing $v$. It can be proved that all the points in a Desarguesian plane are vertices of Wilbrink's condition \RN{2} in strong form; see \cite[Lemma 7.42]{barwick_unitals_2008}. 

We can prove the following result for our unitals.
\begin{proposition}\label{prop:wilbrinkII}
	Let $f$ be a normal planar function on $\F_{q^2}$ and $\U_\theta$ a unital in $\Pi(f)$ defined in Proposition \ref{prop:general}. In $\U_\theta$, $(\infty)$ is a vertex of Wilbrink's condition \RN{2} in strong form.
\end{proposition}
\begin{proof}
	Let $a,b\in \F_{q^2}$ and $B=B_{a,b}$. Let $C$ be an arbitrary block which contains $(\infty)$ and intersects $B$, namely $C=B_c$ for certain $c\in \F_{q^2}$ such that $f(c+a)-b\in \theta \F_q$. Let $w$ be an arbitrary point on $B_c$ which is different from the intersection point $(c,f(c+a)-b)$ of $B_c$ and $B_{a,b}$. Hence $w=(c, d)$ for certain $d\in \theta\F_{q}$ and $d\neq f(c+a)-b$. Then we can take $B'=B_{a,f(c+a)-d}$. 
	
	From $f(c+a)-b\in \theta \F_q$, we see that for any $x\in \F_{q^2}$, $f(x+a)-b\in \theta \F_q$ if and only if $f(x+a)-f(c+a)+d\in \theta \F_q$. Hence $(\infty)$ satisfies the Wilbrink's condition \RN{2} for $B$, $C$ and $w$. From the arbitrariness of $B$, $C$ and $w$, we derive that the point $(\infty)$ is a vertex of Wilbrink's condition \RN{2} in strong form.
\end{proof}

For different shift planes $\Pi(f)$, we cannot find a generic way to investigate whether the affine points in $\U_\theta$ are vertices of Wilbrink's condition \RN{2} or not. For $f(x)=x^2$, we get the following result.
\begin{theorem}\label{th:wilbrink_x^2}
	Let $q$ be a prime power larger than $3$. Let $f(x)=x^2$ be defined on $\F_{q^2}$ and $\U_\theta$ a unital in $\Pi(f)$ defined in Theorem \ref{th:many.planar.functions.1} (a). In $\U_\theta$, $(\infty)$ is the unique vertex of Wilbrink's condition \RN{2} in strong form.
\end{theorem}
\begin{proof}
	From Proposition \ref{prop:wilbrinkII}, we already know that $(\infty)$ is a vertex of Wilbrink's condition \RN{2} in strong form.
	
	Assume that $v$ distinct from $(\infty)$ is a vertex of Wilbrink's condition \RN{2} in strong form. As the shift group is transitive on the affine points in $\U_\theta$, we only have to consider the point $v=(0,0)$. We choose the point $w$, the blocks $B$ and $C$ as follows:
	\begin{eqnarray*}
		w &:=& (0,2\theta),\\
		B &:=& B_1=\{(1, t\theta): t\in \F_q\} \cup \{(\infty)\}\\
		C &:=& B_0=\{(0, t\theta): t\in \F_q\} \cup \{(\infty)\}.
	\end{eqnarray*}
	Clearly $v$ and $w$ are both on $C$, and $C$ meets $B$ in point $(\infty)$. According to Definition \ref{de:wilbrink}, there is a block $B'$ such that 
	\begin{itemize}
		\item $B'$ contains $w$,
		\item $B'$ meets all the blocks which contain $v$ and meet $B$.
	\end{itemize}
	
	The first condition implies that there exists $a\in \F_{q^2}$ such that $B'=B_{a, f(a)-2\theta}$.
	 
	The set $\mathcal{B}$ of all the blocks containing $v$, meeting $B$ and distinct from $C$ is
	\[\mathcal{B}=\{ B_{b, f(b)}: f(1+b)-f(b)=2t\theta, t\in \F_q\}.\]
	In other words, for each point $(1,t\theta)$ on $B$, we have a block $B_{b,f(b)}$ containing $(1,t\theta)$ and $v$. The second condition means that for each $B_{b,f(b)}\in \mathcal{B}$, there is unique $s\in \F_q$ such that
	\begin{eqnarray}
		\label{eq:wilbrink_b_s}	f(x+b)-f(b) &=& s \theta,\\
		\label{eq:wilbrink_a_s} f(x+a) - f(a) + 2\theta &=& s \theta.
	\end{eqnarray}
	Obviously $b\neq a$.
	
	According to the assumption, $f(x)=x^2$. Together with the definition of $B_{b,f(b)}$, \eqref{eq:wilbrink_b_s} and \eqref{eq:wilbrink_a_s} we obtain that
	\begin{equation*}
		b=t\theta-\frac{1}{2} \quad\text{ and }\quad x=\frac{\theta}{b-a}.
	\end{equation*}
	Plugging them back into \eqref{eq:wilbrink_b_s}, we have
	\[\frac{\theta^2}{(b-a)^2} + 2 \left(t\theta -\frac{1}{2}\right) \frac{\theta}{b-a} = s\theta.\]
	Simplifying it, we have
	\begin{equation}\label{eq:wilbrink_(b-a)}
		\frac{\theta}{(b-a)^2} + \frac{2t\theta-1}{b-a}=s.
	\end{equation}
	As $s\in \F_q$, the above equation means
	\[\frac{\theta}{(b-a)^2} + \frac{2t\theta-1}{b-a}= \left(\frac{\theta}{(b-a)^2} + \frac{2t\theta-1}{b-a}\right)^q.\]
	Multiplying both hand sides by $(b-a)^{2q+2}$ and simplifying then yields
	\[\theta (b-a)^{2q} + (2t\theta-1)(b-a)^{2q+1} =  \theta^q (b-a)^{2} + (2t\theta^q-1)(b-a)^{2+q}.\]
	Letting $\bar{a}=a+\frac{1}{2}$ and plugging $b-a=t\theta - \bar{a}$ into the above equation, we obtain
	\begin{align*}
		\theta(t^2 \theta^{2q} + \bar{a}^{2q} - 2\bar{a}^q t\theta^q ) + (2t\theta -1 )(t\theta - \bar{a}) (t^2 \theta^{2q} - 2\bar{a}^q t \theta^q +\bar{a}^{2q})& \\
		-\theta^q(t^2 \theta^{2} + \bar{a}^{2} - 2\bar{a} t\theta ) - (2t\theta^q -1 )(t\theta^q - \bar{a}^q) (t^2 \theta^{2} - 2\bar{a} t \theta +\bar{a}^{2})&=0.
	\end{align*}
	If we view $t$ as an indeterminate and the above equation as an element $g$ in the polynomial ring $\F_{q^2}[t]$, then $g$ vanishes on the elements in $\F_q$. That means the polynomial $t^q-t$ divides $g$. Through straightforward calculation, we see that the coefficient of $t^4$ in $g$ is $0$ and that of $t^3$ is
	\[2\theta^2 (-2 \bar{a}^q\theta^q) - (1+2\bar{a})\theta^{2q+1} - 2\theta^{2q} (-2\bar{a} \theta) + (1+2\bar{a}^q) \theta^{q+2}.\]
	This coefficient must be $0$, because $q>3$ and $(t^q-t) \mid g$. Dividing it by $\theta^{q+1}$ and simplifying, we get 
	\[\theta^q(2\bar{a}-1)= \theta(2\bar{a}^q-1).\] 
	Plugging $\bar{a}=a+\frac{1}{2}$ back into it, we deduce that
	\[2a\theta^q= 2a^q \theta,\]
	which means $a=l\theta$ for certain $l\in \F_q$.
	
	Noting that $t$ can be any element in $\F_q$, we take $t=l$. That means $b-a=-\frac{1}{2}$, and from \eqref{eq:wilbrink_(b-a)} we have
	\[4\theta - 4l\theta +2 =s.\]
	It implies that $l=1$, from which it follows that $a=\theta$.
	
	Now we look at the constant term of $g$, which is
	\[\theta \bar{a}^{2q} + \bar{a}^{2q+1} - \theta^q\bar{a}^2-\bar{a}^{q+2}.\]
	As we showed previously, $(t^q-t) \mid g$. Hence the constant term must be zero. Plugging $\bar{a}=a+\frac{1}{2}=\theta + \frac{1}{2}$ back into it, we have
	\begin{equation}\label{eq:wilbrink_theta_final}
		\theta \left(\theta +\frac{1}{2}\right)^{2q} + \left(\theta +\frac{1}{2}\right)^{2q+1} - \theta^q\left(\theta +\frac{1}{2}\right)^2-\left(\theta +\frac{1}{2}\right)^{q+2}=0.
	\end{equation}
	This can be view as a polynomial in $\theta$ of degree $2q+1$. 
	
	By Proposition \ref{prop:unital_powerplanar_orbit}, we know that all the $\frac{q^2-1}{2}$ unitals $\U_{c^2\theta}$ are equivalent, where $c\in \F_{q^2}^*$. As $\frac{q^2-1}{2}> 2q+1$, we may replace $\theta$ by some $c^2\theta$ such that  \eqref{eq:wilbrink_theta_final} does not hold anymore. This replacement does not affect the previous part of the proof, because we never used the value of $\theta$ there. Therefore, in the unital $\U_{c^2\theta}$ we cannot find a block $B_a$ (because \eqref{eq:wilbrink_theta_final} fails) such that the Wilbrink condition \RN{2} holds.
\end{proof}

\begin{remark}
	For $q=3$, we can use MAGMA \cite{Magma} to show that Theorem \ref{th:wilbrink_x^2} also holds. For other planar functions, it is also possible to investigate the same question. However, if we follow the steps of the proof of Theorem \ref{th:wilbrink_x^2}, we will see the difficulties in solving \eqref{eq:wilbrink_b_s} and \eqref{eq:wilbrink_a_s} as well as more complicated calculations.
\end{remark}

\subsection{Existence of O'Nan configurations}\label{subsec:onan}
An \emph{O'Nan configuration} is a collection of four lines intersecting in six points. In \cite{onan_automorphisms_1972}, O'Nan first considered this configuration, and he proved that it does not exist in the classical unitals.

In \cite{piper_unitary_1979}, Piper conjectured that the nonexistence of O'Nan configuration is also a sufficient condition for a unital to be classical. In \cite{wilbrink_characterization_1983}, Wilbrink investigated this conjecture and obtained a weaker version characterization. See \cite{hui_embedding_2014} for a recent progress on this conjecture.

In this subsection, we consider the existence of O'Nan configuration in the unitals $U_\theta$ in $\Pi(f)$ for various planar functions $f$.

First, we look at O'Nan configurations containing the point $(\infty)$.
\begin{lemma}\label{le:ONan_infty}
	Let $f$ be a planar function from $\F_{q^2}$ to itself and $\theta\in \F_{q^2}^*$ such that $\U_\theta$ is a unital in $\Pi(f)$. There exists an O'Nan configuration with $(\infty)$ as a vertex if and only if there exist $a\in \F_{q^2}^*$ and $\beta \in \F_q^*$ such that $\# (C_{a,\beta} \cap C_{0,1}) \ge 3$, where $C_{a,\beta}$ is the circle defined by \eqref{eq:definition_circles}.
\end{lemma}
\begin{proof}
	$(\Leftarrow)$ By the definition of circles, there exist blocks $B_{a,b}$ and $B_{0,d}$ corresponding to $C_{a,\beta}$ and $C_{0,1}$ respectively for certain $b,d\in \F_{q^2}$, such that 
	\begin{eqnarray*}
		\{x : (x,t\theta)\in B_{a,b} \} &= &C_{a,\beta},\\
		\{x : (x,t\theta)\in B_{0,d} \} &= &C_{0,1}.
	\end{eqnarray*}
	Assume that $u,v,w\in C_{a,\beta} \cap C_{0,1}$, which means that
	\begin{equation}\label{eq:f(u+a)-b_in_F_q}
		f(x+a)-b = f(x) - d \in \theta \F_q,	
	\end{equation}
	for $x=u,v,w$.
	
	Let $d':= f(u) + b - f(u+a)$. It follows that $f(u) - d'= f(u+a) - b\in \theta \F_q$. Together with \eqref{eq:f(u+a)-b_in_F_q}, we have 
	\[d'-d=(d'-f(u))+(f(u)-d) \in \theta \F_q. \]
	Hence $B_{a,b}$ and $B_{0,d'}$ meets at the point $(u,f(u)+d')$ and both $B_{0,d'}$ and $B_{0,d}$ are projected to the same circle $C_{0,1}$. From the assumption, we see that $B_v$ and $B_w$ intersect both $B_{a,b}$ and $B_{0,d'}$. Therefore, these four blocks intersect totally in six points and form an O'Nan configuration.
	
	$(\Rightarrow)$ Assume that there is an O'Nan configuration containing $(\infty)$. It implies that there are two blocks $B$ and $B'$ intersecting in a point $(u,s\theta)$ and not containing $(\infty)$, and they both intersect $B_c$ and $B_{c'}$ for certain $c,c'\in \F_{q^2}$. Under the automorphism group of the unital, we can always assume that $B=B_{0,d}$ and $\Tr_{q^2/q}(d)=1$, which means that the corresponding circle is $C_{0,1}$. Denoting the circle derived from $B'$ by $C_{a,b}$, we see that $c,c', u\in C_{a,\beta} \cap C_{0,1}$.
\end{proof}

According to \cite[Page 80]{baker_buekenhout-metz_1992}, there is no O'Nan configuration with $(\infty)$ in the Buekenhout-Metz unitals of odd order. It implies that there is no O'Nan configuration containing $(\infty)$ in $\U_\theta$ defined over $\Pi(f)$, where $f(x)=x^2$. It is also not difficult to see this result from Lemma \ref{le:ONan_infty} and \eqref{eq:definition_circles}. For two arbitrary circles $C_{a,\beta}$ and $C_{a',\beta'}$, we look at the common solution of $\Tr_{q^2/q}(\delta (x+a)^2) = \beta$ and $\Tr_{q^2/q}(\delta (x+a')^2) = \beta'$. Subtracting the second equation from the first one, we get $\Tr_{q^2/q}(2\delta (a-a')x)=\Tr_{q^2/2}(\delta (a'^2-a^2)) + \beta- \beta'$. Hence if $a\neq a'$, then $x^q$ equals a polynomial of degree one in $\F_{q^2}[x]$. Plugging it back into $\Tr_{q^2/q}(\delta (x+a)^2) = \beta$, we get a quadratic equation, which means there are at most two solutions. The same approach also works for the unital defined in Theorem \ref{th:many.planar.functions.1} when $f$ is derived from a Dickson's semifield.

For $f(x)=x^{p^k+1}$ and $f(x)= x^{\frac{3^k+1}{2}}$, we conjecture that there do exist O'Nan configuration containing the point $(\infty)$. This is confirmed for several small $q$ by using MAGMA \cite{Magma}. However we cannot find a proof. 

Next, let us look for O'Nan configurations without $(\infty)$.
\begin{lemma}\label{le:ONan_without_infty}
	Let $f$ be a planar function from $\F_{q^2}$ to itself and $\theta\in \F_{q^2}^*$ such that $\U_\theta$ is a unital in $\Pi(f)$. Let $x*y := f(x+y)-f(x)-f(y)$. Let $t_u$, $t_v$ and $t_w\in \F_q$ and $a_u$, $a_v$, $a_w\in \F_{q^2}$ which are pairwise distinct. Assume that there exist three distinct elements $x_u$, $x_v$ and $x_w\in \F_{q^2}$ such that for each $k\in \{u,v,w\}$,
	\begin{enumerate}[label=(\alph*)]
		\item $x_{k}*a_i-x_{k}*a_j = (t_j-t_i)\theta$ for pairwise distinct $i$, $j\in \{u,v,w\}$, where $i,j\neq k$, and
		\item $f(x_{k}+a_i)-f(a_i)\in \theta\F_q$ for $i\in \{u,v,w\}\setminus{k}$.
	\end{enumerate}
	Then there is an O'Nan configuration without $(\infty)$ in $\U_\theta$.
\end{lemma}
\begin{proof}
	Under the assumption, we show that there exists an O'Nan configuration containing the points $U=(0,t_u\theta)$, $V=(0,t_v\theta)$, $W=(0,t_w\theta)$. Let $L_i:=\{(x,f(x+a_i)-f(a_i)+t_i\theta ): x\in \F_{q^2} \}$ for $i\in \{u,v,w\}$. Clearly $L_u$, $L_v$ and $L_w$ are the lines in $\Pi(f)$ containing $U$, $V$ and $W$ respectively. The first coordinate $x_w$ of the intersection point of $L_u$ and $L_v$ is the solution of
	\[f(x+a_u)-f(a_u)+t_u\theta=f(x+a_v)-f(a_v)+t_v\theta,\]
	from which it follows that $x_{w}*a_u-x_{w}*a_v=(t_v-t_u)\theta$. According to the definition of $\U_\theta$, this point is in the unital $\U_\theta$ if and only if $f(x_w+a_u)-f(x_w)\in \theta\F_q$.  Similar results can be obtained for $L_u\cap L_w$ and $L_v\cap L_w$. Noting that $x_u$, $x_v$ and $x_w$ are pairwise distinct, $L_u$, $L_v$ and $L_w$ do not intersect in the same point. Therefore these three intersecting points together with $U$, $V$ and $W$ form a O'Nan configuration.
\end{proof}
\begin{remark}
	In Lemma \ref{le:ONan_without_infty}, the O'Nan configuration obtained contains a line through $(\infty)$. Without loss of generality, we can assume that $t_w=0$. From its proof, it is not difficult to see that for given $k$, we only have to take one $i\in \{u,v,w\}\setminus{k}$ to check whether (b) holds.
\end{remark}

Using Lemma \ref{le:ONan_without_infty}, we can prove the existence of O'Nan configurations in $\U_\theta$ for several different planar functions $f$.
\begin{theorem}\label{th:ONan_in_x^d}
	Let $n$ be a positive integer, $q=p^{2n}$ and $k$ an integer satisfying $1\le k \le n$ and $2\nmid \frac{2n}{\gcd(2n,k)}$. For the planar functions $f(x)=x^2$ and $f(x)=x^{p^k+1}$ on $\F_{q^2}$, the unital $\U_\theta$ defined in Theorem \ref{th:many.planar.functions.1} contains an O'Nan configuration.
\end{theorem}
\begin{proof}
	To unify the proof for these two classes of planar functions, we set $k=2n$ when $f(x)=x^2$.
	
	Let $\omega$ be an element in $\F_{q^2}$ such that $\omega^2-\omega+1=0$. Since $2\mid k$, $\omega$ is also in $\F_{p^k}$.  Let $a_v$ and $a_w$ be two elements in $\F_{p^{\gcd(2n,k)}}^*$ such that $a_v\neq a_w$ and $a_v\neq \omega a_w$. Define
	\begin{eqnarray*}
	 a_u &:=& a_w(1-\omega)+\omega a_v,\\
	 t_u &:=& 4a_w(a_v-a_w)\omega/\theta,\\
	 t_v &:=& 4a_v(a_v-a_w)/\theta,\\
	 t_w &:=& 0.
	\end{eqnarray*}
	As $\omega$ is also in $\F_{p^{\gcd(2n,k)}}$, we have that $a_u\in \F_{p^{\gcd(2n,k)}}$. From $a_v\neq a_w$ and $\omega a_w$, it is readily to deduce that $a_u$ is also distinct from $0$, $a_v$ and $a_w$.  As $2\mid k$, we have $\F_{p^{\gcd(2n,k)}}\ge p^2\ge 9$, which guarantees the distinct values of $0$, $a_v$, $a_w$ and $a_u$.
	
	By Lemma \ref{le:ONan_without_infty}, now $x*y=x^{p^k}y+xy^{p^k}$ and 
	\[x_{w}*a_u-x_{w}*a_v = x_w*(a_u-a_v)= (x_w^{p^k}+x_w)(a_u-a_v)=(t_v-t_u)\theta.\] 
	Hence
	\[x_w^{p^k}+x_w = \theta\frac{t_v-t_u}{a_u-a_v}=\frac{4(a_w-a_v)(a_w \omega - a_v)}{(a_w-a_v)(1-\omega)}=-4a_u\neq 0,\]
	where the last equality comes from the fact that $\omega^2-\omega+1=0$. As $a_u\in \F_{p^k}$ and the mapping $x_w\mapsto x_w^{p^k}+x_w$ is a bijection on $\F_{q^2}$, we have $x_w=-2a_u$. Thus $f(x_w+a_u)-f(a_u)=f(-a_u)-f(a_u)=0\in \theta\F_q$. 
	
	Similarly, we have $x_u=-2a_v$ and $x_v=-2a_w$. Thus $f(x_u+a_v)-f(a_v)=f(x_v+a_w)-f(a_w)=0\in \theta\F_q$. 
	
	Therefore, the two conditions in Lemma \ref{le:ONan_without_infty} are satisfied and there is an O'Nan configuration in $\U_\theta$.
\end{proof}

\begin{theorem}\label{th:ONan_in_two_dim}
	Let $q=p^n$ be an odd prime power satisfying that $q \equiv 1 \pmod 4$ and $\xi\in \F_{q^{2}}\setminus \F_{q}$.  Let $\U_\xi$ be a unital defined in Theorem \ref{th:many.planar.functions} for Dickson's semifields or for the semifields constructed in \cite{zhou_new_2013}.  
	If there is $\omega\in \F_{q}$ such that $\omega^2-\omega+1=0$, i.e.\ $2\mid n$, then there is an O'Nan configuration in $\U_\theta$.
\end{theorem}
\begin{proof}
	To unify the proof for these two semifields, we write $f(x)=(x_0^{p^k+1}+\alpha x_1^{p^{k+i}+p^i})+2x_0x_1\xi$, where $0< i, k\le n$ and $n/\gcd(n,k)$ is odd. It implies that the multiplication $x*y$ defined in Lemma \ref{le:ONan_without_infty} is
	\[x*y=(x_0^{p^k}y_0+y_0^{p^k}x_0 + \alpha (x_1^{p^k}y_1 + y_1^{p^k}x_1 )^{p^i}) + (x_0y_1+x_1y_0)\xi\] 
	for $x,y\in \F_{q^2}$.
	 
	
	Similarly as in Theorem \ref{th:ONan_in_x^d}, we take $a_v$ and $a_w\in \F_{p^{\gcd(k,n)}}^*$ such that $a_v\neq a_w$ and $a_v\neq \omega a_w$. It is not difficult to see that $p^{\gcd(k,n)}\ge 5$. Thus we can always find $a_v$ and $a_w$ satisfying all the assumptions. Define
	\begin{eqnarray*}
		a_u &:=& a_w(1-\omega)+\omega a_v,\\
		t_u &:=& 4a_w(a_v-a_w)\omega/\xi,\\
		t_v &:=& 4a_v(a_v-a_w)/\xi,\\
		t_w &:=& 0.
	\end{eqnarray*}
	It is straightforward to verify the conditions in Lemma \ref{le:ONan_without_infty}.
\end{proof}

\subsection{Open problems and remarks}
To conclude this section, we propose an open problem on the automorphism group $\Aut(\U_\theta)$ of $\U_\theta$.

\begin{conjecture}\label{conjecture:automorphism.main}
	Let $f:\F_{q^2}\rightarrow \F_{q^2}$ be a normal planar function  and  $\U_\theta$ a unital in $\Pi(f)$ defined in Proposition \ref{prop:general}. Then
	\[\Aut(\U_\theta)=\Aut_{\Pi(f)}(\U_\theta).\]
\end{conjecture}

In general, it is quite difficult to determine the automorphism group of a unital as a design. For a classical unital in $\mathrm{PG}(2,q^2)$, O'Nan \cite{onan_automorphisms_1972} proved that its automorphism group is $\mathrm{P\Gamma U}(3,q)$. In \cite{hui_non-classical_2013}, Hui, Law, Tai and Wong obtained a similar result for the unitals defined by polarities in the Dickson's semifield planes.

A natural approach, which was used in \cite{baker_buekenhout-metz_1992,hui_non-classical_2013,onan_automorphisms_1972}, is to consider the automorphism group of the design formed by the set of all circles derived/projected from the blocks of the corresponding unital. In \cite{baker_buekenhout-metz_1992,hui_non-classical_2013}, one extra point was added to this design and a miquelian inversive plane was obtained. As the automorphism group of a miquelian inversive plane is $\mathrm{P\Gamma L}(2,q)$ (see \cite{dembowski_finite_1997}), it is possible to further determine the automorphism group of the original unital.

Unfortunately, this approach cannot be generally applied on the unitals $\U_\theta$ constructed in this paper. For example, when $f(x)$ is a Coulter-Matthews planar function, it is possible that there exists an O'Nan configuration containing $(\infty)$ (It can be verified by using MAGMA for $q=3^2$ and $f(x)=x^{14}$). From Lemma \ref{le:ONan_infty}, it implies that there exist three elements $\{x_1,x_2,x_3\}$ belonging to two cycles. That means we cannot extend the design formed by all circles into an inversive plane.

\section{Inequivalence between $\U_\theta$ and the unitals derived from polarities}\label{sec:inequivalence}

A \emph{correlation} $\rho$ of a projective plane $\Pi$ is a one-to-one map of the points onto the lines and the lines onto the points such that the point $P$ is on the line $l$ if and only if $\rho(l)$ is on $\rho(P)$. A \emph{polarity} is a correlation of order two.

For each polarity $\rho$ on a projective plane $\Pi$, a point $P$ is called \emph{absolute} if $P$ is on the line $\rho(P)$. When the plane $\Pi$ is of order $q^2$, if the polarity $\rho$ has $q^3+1$ absolute points, then $\rho$ is a \emph{unitary} polarity. For each unitary polarity of $\Pi$, the set of absolute points and non-absolute lines forms a unital; see \cite[Theorem 12.12]{hughes_projective_1973}.

\begin{lemma}\label{le:polarity}
	Let $f$ be a planar function on $\F_{q^2}$ and $\kappa$ (also denoted by $x\mapsto \bar{x}$ for convenience) an additive map on $\F_{q^2}$ such that
	\begin{enumerate}[label=(\alph*)]
		\item $\kappa$ is involutionary, i.e.\ $\kappa(\kappa(x))=x$,
		\item $f$ and $\kappa$ are commutative under the composition, i.e.\ $\overline{f(x)}=f(\bar{x})$,
		\item $\#\{y:y+\bar{y}=f(x+\bar{x}) \}=q$ for each $x\in \F_{q^2}$.
	\end{enumerate}
	Let $\rho$ be a map from the points of $\Pi(f)$ to its line sets defined by:
	\begin{eqnarray*}
		(x,y) & \mapsto & L_{\bar{x},\bar{y}},\\
		(a) & \mapsto & N_{\bar{a}},\\
		(\infty) & \mapsto & L_{\infty} .
	\end{eqnarray*}
	Then $\rho$ induces a unitary polarity on $\Pi(f)$ and the corresponding unital is
	\begin{equation}\label{eq:unital_polarity}
		\U := \{(x,y): y + \bar{y} = f(x+\bar{x})  \} \cup \{ (\infty) \}.
	\end{equation}
\end{lemma}
\begin{proof}
	First we show that $\rho$ induces a polarity. Let us look at the points on $L_{\bar{x},\bar{y}}$, which are $\{(u, f(\bar{x}+u)-\bar{y} ): u\in \F_{q^2} \} \cup \{(\bar{x})\}$. Under $\rho$, they are mapped to
	\[\{L_{\bar{u}, \overline{f(\bar{x}+u)-\bar{y}}}: u\in \F_{q^2}  \} \cup \{ N_{x} \}. \]
	Under the assumption (a) and (b), the line $L_{\bar{u}, \overline{f(\bar{x}+u)-\bar{y}}}=L_{\bar{u}, f(x+\bar{u})-y}$. 
	All of these lines intersect in the point $(x,y)$. Hence $\rho(\rho(x,y))=(x,y)$. Similarly, we can prove that $\rho(N_{\bar{a}})= (a)$ and  $\rho(L_{\infty})= (\infty)$. Therefore, $\rho$ defines a polarity on $\Pi(f)$.
	
	It is not difficult to check that $(\infty)$ is an absolute point of $\rho$ and for each $a\in \F_{q^2}$ the point $(a)$ is not absolute. For each affine point $(x,y)$, it is absolute if and only if $(x,y)$ is on $L_{\bar{x}, \bar{y}}$, i.e.\ 
	\[y= f(x+\bar{x})-\bar{y}.\]
	Hence the set defined by \eqref{eq:unital_polarity} is the set of absolute points of $\rho$. From the assumption (c), we know that there are totally $q\cdot q^2+1$ points in $\U$. Therefore, $\U$ is a unital.
\end{proof}

When $\Pi(f)$ is a commutative semifield plane of order $q^2$, i.e.\ $f$ can be written as a Dembowski-Ostrom polynomial over $\F_{q^2}$, the unitals derived from unitary polarities on $\Pi(f)$ are intensively investigated by Ganley in \cite{ganley_polarities_1972}. For the Coulter-Matthews planes, the unitary polarities and derived unitals are considered by Knarr and Stroppel in \cite{knarr_polarities_2010}.

For the power planar functions (see Theorem \ref{th:many.planar.functions.1}), we take $\kappa(x)=\bar{x}=x^q$ for $x\in \F_{q^2}$; for the Dickson's semifields and the one constructed in \cite{zhou_new_2013} (see Theorem \ref{th:many.planar.functions} (a) and (b)), we set $\kappa(x)=\bar{x}=x_0-x_1\xi $ for $x=x_0+x_1\xi\in \F_{q^2}$. It is readily to verify that for these two maps $\kappa$, the three assumptions in Lemma \ref{le:polarity} are satisfied.

Recall that two unitals $\U_1$ and $\U_2$ in a projective plane $\Pi$ are equivalent if there is a collineation of $\Pi$ mapping $\U_1$ to $\U_2$. One of the goals of this section is to prove the following theorem.

\begin{theorem}\label{th:inequivalence_unitals_semifields}
	Let $f$ be a Dembowski-Ostrom polynomial over $\F_{q^2}$ such that  $x\mapsto f(x)$ is a planar function on $\F_{q^2}$. Let $\kappa$ be a map on $\F_{q^2}$ satisfying all the assumptions in Lemma \ref{le:polarity} and $\U$ the unital derived from $\kappa$. Let $\U_\theta$ be a unital defined in Proposition \ref{prop:general}. Then $\U$ and $\U_\theta$ are not equivalent.
\end{theorem}

To prove Theorem \ref{th:inequivalence_unitals_semifields}, we proceed to consider the collineation group of $\Pi(f)$. Slightly different from  Lemma \ref{le:ONan_without_infty}, we define $x*y := \frac{1}{2}(f(x+y)-f(x)-f(y))$. It is readily to verify that $x*x=f(x)$ if $f$ is a Dembowski-Ostrom polynomial.
\begin{lemma}\label{le:translation+shear_group}
	Let $u,v,w\in \F_{q^2}$. Let $f$ be a Dembowski-Ostrom polynomial over $\F_{q^2}$ which defines a planar function. Let the map $\varsigma_{u,v,w}$ on the points of $\Pi(f)$ be defined by
	\begin{eqnarray*}
		(x,y) & \mapsto & (x+u,y+ 2w*x -v),\\
		(x) & \mapsto & (x+w-u),\\
		(\infty) & \mapsto & (\infty),
	\end{eqnarray*}
	for each $x,y\in \F_{q^2}$. Then the following statements hold.
	\begin{enumerate}[label=(\alph*)]
		\item All $\varsigma_{u,v,w}$ form a collineation group $\Sigma$ of order $q^6$. For $u,v,w,u', v',w'\in \F_{q^2}$,
		\begin{equation}\label{eq:multiplication_varsigma}
			\varsigma_{u',v',w'}\circ \varsigma_{u,v,w} = \varsigma_{u+ u',v+v'-2w'*u,w+w'}
		\end{equation}
		\item The shift group $T=\{\varsigma_{u,v,0}:u,v\in \F_{q^2}\}$.
		\item $\{\varsigma_{u,v,u}:u,v\in \F_{q^2}\}$ is the translation group of $\Pi(f)$.
		\item $\{\varsigma_{0,0,w}:w\in \F_{q^2}\}$ is the group of elations with the axis $N_0$ and the center $(\infty)$.
	\end{enumerate}
\end{lemma}
\begin{proof}
	For any $u,v,w,a,b$ and $x\in \F_{q^2}$,
		\[\varsigma_{u,v,w} :(x, f(x+a)-b)\mapsto (x+u, f(x+a)+2 w*x-b+v) \]
	which equals $(x+u, f(x+a+w)-f(w)-b+v)$.
	Together with $\varsigma_{u,v,w}(x)\mapsto (x+w-u)$, we see that the line $L_{a,b}$ is mapped to $L_{a+w-u,b+f(w)-v}$ under $\varsigma_{u,v,w}$. Similarly, we can show that $L_\infty$ is fixed and $N_a$ is mapped to $N_{a+u}$. Hence $\varsigma_{u,v,w}$ is a collineation.
	
	Moreover, it is routine to verify \eqref{eq:multiplication_varsigma} and that all $\varsigma_{u,v,w}$ form a collineation group $\Sigma$. 
	
	From (a) and $\varsigma_{u,v,w}(L_{a,b})=L_{a+w-u,b+f(w)-v}$, it is not difficult to derive (b), (c) and (d).
\end{proof}

The following lemma can be found in \cite[Lemma 8.5]{hughes_projective_1973}.
\begin{lemma}\label{le:normal_sigma}
	With the notation in Lemma \ref{le:translation+shear_group}, $\Sigma$ is a normal subgroup of $\Aut(\Pi(f))$.
\end{lemma}

Now let us prove Theorem \ref{th:inequivalence_unitals_semifields}.
\begin{proof}
	Let us first look at the elements in $\Sigma$ fixing $\U_\theta$. If $\varsigma_{u,v,w}$ fixes $\U_\theta$, then $w=0$. Otherwise, for arbitrary $y\in \F_{q^2}$ and $t\in \F_q$, there is always a unique solution $x$ such that
	\[t\theta + 2 w* x-v=y.\]
	It implies that $(x+u, t\theta+2w*x-v)\not\in \U_\theta$ when we choose $y\in \F_{q^2}\setminus {\theta}\F_q$.
	
	Furthermore, it is also readily to verify that $\varsigma_{u,v,0}$ fixes $\U_\theta$ if and only if $v\in \theta\F_q$. Therefore, the subgroup in $\Sigma$ fixing $\U_\theta$ is 
	\[\Sigma_1:=\{\varsigma_{u,v,0}: u\in\F_{q^2}, v\in \theta\F_q \}, \]
	which is abelian.
	
	Second, we consider the elements in $\Sigma$ fixing $\U$. As there is only the point $(\infty)$ on $L_\infty$ which is fixed by $\Sigma$, we only have to consider the affine points in $\U$. Let $(x,y)$ be an arbitrary point in $\U$, which means that $f(x+\bar{x})=y+\bar{y}$. The point $\varsigma_{u,v,w}(x,y)=(x+u,y+2w*x-v)$ is still in $\U$, which implies that
	\[f(x+u + \bar{x} + \bar{u})= y+2w*x-v + \overline{y+2w*x-v}.\]
	From the above equation,  as well as
	\begin{align*}
		 & f(x+u + \bar{x} + \bar{u})\\
		=& f(x+\bar{x}) + 2(x+\bar{x})*(u+\bar{u})  +f(u+\bar{u})
	\end{align*}
	and
	\begin{align*}
		 & y+2w*x-v + \overline{y+2w*x-v}\\
		=& f(x+\bar{x}) +2w*x+2\bar{w}*\bar{x}-(v+\bar{v}),
	\end{align*}
	we deduce that 
	\begin{equation}\label{eq:fixing_U}
		2(u+\bar{u})*(x+\bar{x}) + f(u+\bar{u}) = 2\bar{w}*\bar{x} + 2w*x-(v+\bar{v}).
	\end{equation}
	As \eqref{eq:fixing_U} holds for arbitrary $x\in \F_{q^2}$ (see the assumption (c) in Lemma \ref{le:polarity}),  we have
	\begin{eqnarray*}
		w &=& u+\bar{u},\\
		-(v+\bar{v}) &=& f(u+\bar{u}).
	\end{eqnarray*}
	Hence the subgroup in $\Sigma$ fixing $\U$ is
	\[\Sigma_2 := \{\varsigma_{u,v,u+\bar{u}}:  f(u+\bar{u})=-(v+\bar{v})  \}.\]
	From the assumption (c) in Lemma \ref{le:polarity}, there are totally $q^3$ elements in $\Sigma_2$. By Lemma \ref{le:translation+shear_group} (a), we see that $\Sigma_2$ is not abelian.
	
	Assume to the contrary that $\U_\theta$ is equivalent to $\U$, i.e.\ there exists $\varphi\in \Aut(\Pi(f))$ such that $\varphi(\U_\theta)=\U$. As $\Sigma$ is normal in $\Aut(\Pi(f))$, the two subgroups $\Sigma_1$ and $\Sigma_2$ are conjugate, which contradicts the fact that $\Sigma_1$ is abelian but $\Sigma_2$ is not.
\end{proof}
%

Clearly Theorem \ref{th:inequivalence_unitals_semifields} does not cover the Coulter-Matthews planar functions. To prove a similar result for the inequivalence between the two unitals $\U$ and $\U_\theta$, we need the following result by Dempwolff and R\"{o}der in \cite{dempwolff_finite_2006}.
\begin{theorem}\label{th:aut.CM.plane}
	Let $\Pi(x^d)$ be a Coulter-Matthews plane of order $3^m$. Then its automorphism group is
	\[\Aut(\Pi(x^d)) =T \rtimes\Gamma, \]
	where $T\cong(\F_{3^{m}},+)^2$ and $\Gamma \cong \Gamma L (1,3^{m})$.
\end{theorem}

In Theorem \ref{th:aut.CM.plane}, the group $T$ is the shift group (see Section \ref{sec:construction}). The definition of $\Gamma$ can be found in Proposition \ref{prop:unital_powerplanar_orbit}.

Using Theorem \ref{th:aut.CM.plane}, we can prove the following results.
\begin{theorem}\label{th:inequivalence_unitals_CM}
	Let $m=2n$ and $\Pi(x^d)$ a Coulter-Matthews plane of order $3^m$. The subgroups of $T$ respectively fixing $\U$ and $\U_\theta$ are of size $3^{2n}$ and $3^{3n}$. Therefore $\U$ and $\U_\theta$ are inequivalent.
\end{theorem}
\begin{proof}
	Similarly as in the proof of Theorem \ref{th:inequivalence_unitals_semifields}, $\tau_{a,b}$ fixes $\U_\theta$ if and only if $b=s\theta$ for some $s\in \F_{3^m}$. We denote the group of all these $3^{3n}$ maps $\tau_{a,b}$ by $\Sigma_1$, which is also abelian.
	
	It is routine to show that $\tau_{a,b}$ fixes $\U$ if and only if $a+\bar{a}=0$ and $b+\bar{b}=0$. Hence there are totally $3^{2n}$ such elements in $T$. They form a group denoted by $\Sigma_2$.  
	
	As $T$ is a normal subgroup of $\Aut(\Pi(x^d))$, if $\U$ can be mapped to $\U_\theta$ under a collineation of $\Pi(x^d)$, then $\Sigma_2$ must be conjugate to $\Sigma_1$ which leads to a contradiction.
\end{proof}

To conclude this section, we proceed to consider the isomorphism between the classical unitals (Hermitian curves) $\mathcal{H}$ in $\mathrm{PG}(2,q^2)$ and $\U_\theta$ in any shift plane $\Pi(f)$. We use $\Aut(\mathcal{H})$ to denote the automorphism groups of the classical unitals $\mathcal{H}$ as designs. In \cite{onan_automorphisms_1972}, O'Nan proved that  
\[\Aut(\mathcal{H})\cong \mathrm{P\Gamma U}(3,q),\]
which is transitive on the point sets of $\mathcal{H}$.

Let $p=\mathrm{char}(\F_q)$ and $n$ is such that $q=p^n$. The order of $\Aut(\mathcal{H})$ is $2n(q^3+1)q^3(q^2-1)$ (\cite[Page 496 and 503]{onan_automorphisms_1972}). In \cite[Theorem 1]{baker_buekenhout-metz_1992}, Baker and Ebert proved that for a given point $(\infty)$ in $\mathcal{H}$, the collineation group $G$ of $PG(2,q^2)$ fixing $\mathcal{H}$ and $(\infty)$ is of order $2n q^3(q^2-1)$. Hence $G$ is exactly the stabilizer of $(\infty)$ in $\Aut(\mathcal{H})$. In $G$, there is a non-abelian subgroup $S$ of order $q^3$ and a subgroup $R$ corresponding to the field automorphism group of $\Gal(\F_{q^2}/\F_p)$, $R$ and $S$ intersect at the identity element of $G$ and $R$ normalize $S$; see \cite{baker_buekenhout-metz_1992} and \cite[Page 70]{barwick_unitals_2008}. 

From the proofs of Theorems \ref{th:inequivalence_unitals_semifields} and \ref{th:inequivalence_unitals_CM}, we know that in the stabilizer of $(\infty)$ in the group $\Aut(\U_\theta)$, there is an abelian subgroup $\Sigma_1$ of order $q^3$. That means if there is no abelian subgroup of order $q^3$ in $G$, then $\mathcal{H}$ and $\U_\theta$ can never be isomorphic as designs.

When $\gcd(n,p)=1$, it is clear that the non-abelian group $S$ is a Sylow $p$-subgroup of $G$. By Sylow theorems, all subgroups of order $q^3$ in $G$ are non-abelian. 

When $\gcd(n,p)=p$, it is also not difficult to follow the calculations in \cite[Page 70]{barwick_unitals_2008} to verify that there is no abelian subgroups of order $q^3$ in $S\rtimes R$. Hence there is also no abelian subgroups of order $q^3$ in $G$.

Therefore, we have proved the following result.

\begin{theorem}
	Let $f$ be a planar function on $\F_{q^2}$. If $f$ can be written as a Dembowski-Ostrom polynomial over $\F_{q^2}$ or $f$ is the Coulter-Matthews function. As designs, the unitals $\U_\theta$ defined by \eqref{eq:u_theta_general} in the plane $\Pi(f)$  and the classical unitals $\mathcal{H}$ are not isomorphic.
\end{theorem}

\section*{Acknowledgment}
The authors would like to thank Alexander Pott for his valuable comments and suggestions. This work is supported by the Research Project of MIUR (Italian Office for University and Research) ``Strutture geometriche, Combinatoria e loro Applicazioni" 2012.


\begin{thebibliography}{10}

\bibitem{abatangelo_ovals_1999}
V.~Abatangelo, M.~R. Enea, G.~Korchm\'{a}ros, and B.~Larato.
\newblock Ovals and unitals in commutative twisted field planes.
\newblock {\em Discrete Mathematics}, 208{\textendash}209:3{\textendash}8,
  1999.

\bibitem{abatangelo_transitive_2001}
V.~Abatangelo, G.~Korchm\'{a}ros, and B.~Larato.
\newblock Transitive parabolic unitals in translation planes of odd order.
\newblock {\em Discrete Mathematics}, 231(1{\textendash}3):3--10, Mar. 2001.

\bibitem{abatangelo_polarity_2002}
V.~Abatangelo and B.~Larato.
\newblock Polarity and transitive parabolic unitals in translation planes of
  odd order.
\newblock {\em Journal of Geometry}, 74(1-2):1--6, Nov. 2002.

\bibitem{albert_generalized_1961}
A.~A. Albert.
\newblock Generalized twisted fields.
\newblock {\em Pacific Journal of Mathematics}, 11:1{\textendash}8, 1961.

\bibitem{baker_intersection_1990}
R.~D. Baker and G.~L. Ebert.
\newblock Intersection of unitals in the {Desarguesian} plane.
\newblock In {\em Congressus {Numerantium}. {A} {Conference} {Journal} on
  {Numerical} {Themes}}, volume~70, pages 87--94, 1990.

\bibitem{baker_buekenhout-metz_1992}
R.~D. Baker and G.~L. Ebert.
\newblock On buekenhout-metz unitals of odd order.
\newblock {\em Journal of Combinatorial Theory, Series A}, 60(1):67--84, May
  1992.

\bibitem{barwick_unitals_2008}
S.~Barwick and G.~Ebert.
\newblock {\em Unitals in Projective Planes}.
\newblock Springer Monographs in Mathematics. Springer New York, Jan. 2008.

\bibitem{blokhuis_proof_2002}
A.~Blokhuis, D.~Jungnickel, and B.~Schmidt.
\newblock Proof of the prime power conjecture for projective planes of order
  $n$ with abelian collineation groups of order $n^2$.
\newblock {\em Proceedings of the American Mathematical Society},
  130(5):1473{\textendash}1476, 2002.

\bibitem{Magma}
W.~Bosma, J.~Cannon, and C.~Playoust.
\newblock The {MAGMA} algebra system {I}: the user language.
\newblock {\em J. Symb. Comput.}, 24(3-4):235{\textendash}265, 1997.

\bibitem{budaghyan_new_2008}
L.~Budaghyan and T.~Helleseth.
\newblock New perfect nonlinear multinomials over {$\F_{p^{2k}}$} for any odd
  prime $p$.
\newblock In {\em SETA '08: Proceedings of the 5th international conference on
  Sequences and Their Applications}, page 403{\textendash}414, Berlin,
  Heidelberg, 2008. Springer-Verlag.

\bibitem{buekenhout_characterizations_1976}
F.~Buekenhout.
\newblock Characterizations of semi quadrics.
\newblock {\em Atti Conv. Lincei}, 17:393--421, 1976.

\bibitem{coulter_planar_1997}
R.~S. Coulter and R.~W. Matthews.
\newblock Planar functions and planes of {L}enz-{B}arlotti class {II}.
\newblock {\em Des. Codes Cryptography}, 10(2):167{\textendash}184, 1997.

\bibitem{dembowski_finite_1997}
P.~Dembowski.
\newblock {\em Finite Geometries}.
\newblock Springer, 1997.

\bibitem{dembowski_planes_1968}
P.~Dembowski and T.~G. Ostrom.
\newblock Planes of order $n$ with collineation groups of order $n^2$.
\newblock {\em Mathematische Zeitschrift}, 103:239--258, 1968.

\bibitem{dempwolff_finite_2006}
U.~Dempwolff and M.~R{\"o}der.
\newblock On finite projective planes defined by planar monomials.
\newblock {\em Innovations in Incidence Geometry}, 4:103{\textendash}108, 2006.

\bibitem{dickson_commutative_1906}
L.~E. Dickson.
\newblock On commutative linear algebras in which division is always uniquely
  possible.
\newblock {\em Transactions of the American Mathematical Society},
  7(4):514{\textendash}522, 1906.

\bibitem{durante_unitals_2013}
N.~Durante and A.~Siciliano.
\newblock Unitals of $\mathrm{PG}(2,q^2)$ containing conics.
\newblock {\em Journal of Combinatorial Designs}, 21(3):101--111, Mar. 2013.

\bibitem{ganley_class_1972}
M.~J. Ganley.
\newblock A class of unitary block designs.
\newblock {\em Mathematische Zeitschrift}, 128(1):34--42, Mar. 1972.

\bibitem{ganley_polarities_1972}
M.~J. Ganley.
\newblock Polarities in translation planes.
\newblock {\em Geometriae Dedicata}, 1(1):103{\textendash}116, 1972.

\bibitem{ganley_paper_1976}
M.~J. Ganley.
\newblock On a paper of {P. D}embowski and {T. G. O}strom.
\newblock {\em Archiv der Mathematik}, 27(1):93{\textendash}98, 1976.

\bibitem{ganley_central_1981}
M.~J. Ganley.
\newblock Central weak nucleus semifields.
\newblock {\em European Journal of Combinatorics}, 2:339--347, 1981.

\bibitem{ganley_relative_1975}
M.~J. Ganley and E.~Spence.
\newblock Relative difference sets and quasiregular collineation groups.
\newblock {\em Journal of Combinatorial Theory. Series A},
  19(2):134{\textendash}153, 1975.

\bibitem{ghinelli_finite_2003}
D.~Ghinelli and D.~Jungnickel.
\newblock Finite projective planes with a large abelian group.
\newblock In {\em Surveys in combinatorics, 2003 {(Bangor)}}, volume 307 of
  {\em London Math. Soc. Lecture Note Ser.}, page 175{\textendash}237.
  Cambridge Univ. Press, Cambridge, 2003.

\bibitem{hirschfeld_sets_1991}
J.~W.~P. Hirschfeld and T.~Sz\"{o}nyi.
\newblock Sets in a finite plane with few intersection numbers and a
  distinguished point.
\newblock {\em Discrete Mathematics}, 97(1{\textendash}3):229--242, Dec. 1991.

\bibitem{hughes_projective_1973}
D.~R. Hughes and F.~C. Piper.
\newblock {\em Projective planes}.
\newblock Springer-Verlag, New York, 1973.
\newblock Graduate Texts in Mathematics, Vol. 6.

\bibitem{hui_non-classical_2013}
A.~M.~W. Hui, H.~F. Law, Y.~K. Tai, and P.~P.~W. Wong.
\newblock Non-classical polar unitals in finite dickson semifield planes.
\newblock {\em Journal of Geometry}, 104(3):469--493, Dec. 2013.

\bibitem{hui_embedding_2014}
A.~M.~W. Hui and P.~P.~W. Wong.
\newblock On embedding a unitary block design as a polar unital and an
  intrinsic characterization of the classical unital.
\newblock {\em Journal of Combinatorial Theory, Series A}, 122:39--52, Feb.
  2014.

\bibitem{knarr_polarities_2009}
N.~Knarr and M.~Stroppel.
\newblock Polarities of shift planes.
\newblock {\em Advances in Geometry}, 9(4):577{\textendash}603, Aug. 2009.

\bibitem{knarr_polarities_2010}
N.~Knarr and M.~Stroppel.
\newblock Polarities and unitals in the {C}oulter-{M}atthews planes.
\newblock {\em Des. Codes Cryptography}, 55(1):9--18, Apr. 2010.

\bibitem{kyureghyan_theorems_2008}
G.~M. Kyureghyan and A.~Pott.
\newblock Some theorems on planar mappings.
\newblock In {\em Arithmetic of finite fields}, volume 5130 of {\em Lecture
  Notes in Comput. Sci.}, page 117{\textendash}122. Springer, Berlin, 2008.

\bibitem{lavrauw_semifields_2011}
M.~Lavrauw and O.~Polverino.
\newblock Finite semifields.
\newblock In L.~Storme and J.~De~Beule, editors, {\em Current research topics
  in {G}alois Geometry}, chapter~6, pages 131--160. {NOVA} Academic Publishers,
  2011.

\bibitem{luneburg_remarks_1966}
H.~L\"{u}neburg.
\newblock Some remarks concerning the {Ree} groups of type {$G_2$}.
\newblock {\em Journal of Algebra}, 3(2):256--259, Mar. 1966.

\bibitem{onan_automorphisms_1972}
M.~E. O'Nan.
\newblock Automorphisms of unitary block designs.
\newblock {\em Journal of Algebra}, 20(3):495--511, Mar. 1972.

\bibitem{penttila_ovoids_2004}
T.~Penttila and B.~Williams.
\newblock Ovoids of parabolic spaces.
\newblock {\em Geometriae Dedicata}, 82(1-3):1--19, November 2004.

\bibitem{piper_unitary_1979}
F.~Piper.
\newblock Unitary block designs.
\newblock In R.~J. Wilson, editor, {\em Graph Theory and Combinatorics},
  volume~34 of {\em Research Notes in Mathematics}, page 98{\textendash}105.
  Pitman Advanced Publishing Program, Boston, 1979.

\bibitem{pott_semifields_2014}
A.~Pott, K.-U. Schmidt, and Y.~Zhou.
\newblock Semifields, relative difference sets, and bent functions.
\newblock In H.~Niederreiter, A.~Ostafe, D.~Panario, and A.~Winterhof, editors,
  {\em Algebraic Curves and Finite Fields, Cryptography and Other
  Applications}. De Gruyter, 2014.

\bibitem{schmidt_planar_2014}
K.-U. Schmidt and Y.~Zhou.
\newblock Planar functions over fields of characteristic two.
\newblock {\em Journal of Algebraic Combinatorics}, 40(2):503--526, Sept. 2014.

\bibitem{segre_ovals_1955}
B.~Segre.
\newblock Ovals in a finite projective plane.
\newblock {\em Canadian Journal of Mathematics}, 7:414--416, Jan. 1955.

\bibitem{weng_further_2010}
G.~Weng and X.~Zeng.
\newblock Further results on planar {DO} functions and commutative semifields.
\newblock {\em Designs, Codes and Cryptography}, 63(3):413{\textendash}423,
  2012.

\bibitem{wilbrink_characterization_1983}
H.~Wilbrink.
\newblock {A characterization of the classical unitals.}
\newblock In {\em Finite geometries}, volume~82 of {\em Lecture Notes in Pure
  and Applied Mathematics}, pages 445--454. Dekker, New York, 1983.

\bibitem{zhou_2^n2^n2^n1-relative_2012}
Y.~Zhou.
\newblock $(2^n,2^n,2^n,1)$-relative difference sets and their representations.
\newblock {\em Journal of Combinatorial Designs}, 21(12):563{\textendash}584,
  2013.

\bibitem{zhou_parabolic_2015}
Y.~Zhou.
\newblock Parabolic unitals in a family of commutative semifield planes.
\newblock {\em Discrete Mathematics}, 338(8):1300--1306, Aug. 2015.

\bibitem{zhou_new_2013}
Y.~Zhou and A.~Pott.
\newblock A new family of semifields with 2 parameters.
\newblock {\em Advances in Mathematics}, 234:43{\textendash}60, Feb. 2013.

\end{thebibliography}
\end{document}